\documentclass[oneside]{amsart}
\usepackage{amsmath,amssymb,amsfonts,amsthm}
\usepackage{color}
\usepackage{graphicx}
\usepackage{geometry} % Paolo: authos,affiliation
\usepackage{amscd}
\usepackage{enumitem}

\title{On the    problem of  non Berwaldian Landsberg  spaces }

\author[Elgendi]{S. G.~Elgendi}

\address{S. G. Elgendi, Department of Mathematics, Faculty of Science, Benha
  University, Egypt} \email{salah.ali@fsci.bu.edu.eg, \, salahelgendi@yahoo.com}
  \urladdr{http://www.bu.edu.eg/staff/salahali7}



\def\blue#1{\textcolor[rgb]{0.0,0.0,1.0}{#1}}

\newcommand{\T}{{\mathcal T}}

\newcommand{\tm}{\T M}

\def\+{\!+\!}

\def\={\!=\!}
\def\<{\!<\!}
\def\>{\!>\!}

\newcommand\undersym[2]{\raisebox{-7pt}{\tiny$#2$}{\kern-8pt}\mbox{$#1$}}
\newcommand\undersymm[2]{\raisebox{-7pt}{\tiny$#2$}{\kern-15pt}\mbox{$#1$}}

 \let\oldmarginpar\marginpar
\renewcommand\marginpar[1]{\oldmarginpar[\raggedleft\footnotesize #1]%
  {\blue{\raggedright \footnotesize \fbox{
      \begin{minipage}{1.0\linewidth}
        #1
      \end{minipage}
}}}}

\numberwithin{equation}{section} %% Comment out for sequentially-numbered
\numberwithin{figure}{section} %% Comment out for sequentially-numbered

\theoremstyle{plain}
\newtheorem*{theorem*}{Theorem}
\newtheorem{theorem}{Theorem}[section]
\newtheorem{lemma}[theorem]{Lemma}
\newtheorem{proposition}[theorem]{Proposition}
\newtheorem{corollary}[theorem]{Corollary}
\theoremstyle{definition}
\newtheorem{definition}[theorem]{Definition}

\theoremstyle{remark}
\newtheorem{example}{Example}
\newtheorem{remark}[theorem]{Remark}
\newtheorem*{acknowledgement*}{Acknowledgement}

\begin{document}

\maketitle

\begin{abstract}
In this paper, we study the long existence  problem of non Berwaldian Landsberg   spaces using the conformal transformation  point of view. Under conformal transformation, the Berwald  and Landesberg tensors are calculated in terms of the T-tensor. By giving examples, we show that under conformal transformation, there are Landsberg spaces with non-vanishing T-tensor. A necessary condition for  a Landsberg space to be Berwaldian is given.   Various special cases are studied. Cases in which the Landsberg spaces can not be Berwaldian  are shown.  Examples of non-Berwaldian landsberg (singular) spaces  are given. 

\end{abstract}

\section{Introduction}

 A Finsler space is a smooth manifold $M$  such that each tangent space $T_xM$ is equipped  with a Minkowski norm.
 In the case that each of the Minkowski norms is a scalar product on the
corresponding tangent space $T_xM$, then $M$ is a Riemannian manifold. If the
slit tangent spaces $T_xM \backslash\{ 0\}$ are endowed with homogeneous Riemannian
metric which are all isometric, $M$ is a Landsberg space; if such isometries are
linear, then $M$ is a Berwald space.
The  regular Landsberg spaces are the most elusive. 
Around 1907  G. Landsberg \cite{Landsberg1,Landsberg2,Landsberg3} introduced the Landsberg spaces in a non-Finsler frame work.  Every Berwald space is a Landsberg space. Whether there
are Landsberg spaces which are not of Berwald type is a long-standing question
in Finsler geometry, which is still open.  Despite the intense effort by many Finsler
geometers, it is not known an example of a regular  non Berwaldian Landsberg space. 

 In \cite{Asanov}, Asanov  obtained  examples of non-Berwaldian Landsberg spaces, of
dimension at least $3$. In Asanov's examples the Finsler
functions are not defined for all values of the fibre coordinates $y^i$ ($y$-local).
Whether or not there are $ y$-global non-Berwaldian Landsberg spaces remains
an open question.  Shen \cite{Shen_example}  studied the class of $(\alpha,\beta)$ metrics of Landsberg type, of which Asanov's examples
are particular cases; he found \cite{Shen_etal,Shen_example} that  there are $y$-local non-Berwaldian
Landsberg spaces with $(\alpha,\beta)$ metrics, there are no $y$-global ones. Bao \cite{Bao} tried to  construct
non-Berwaldian Landsberg spaces by successive approximation; but this method
 so far could not  solve the problem. The elusiveness of $y$-global non-Berwaldian Landsberg spaces leads Bao to describe them as the unicorns of Finsler geometry. 

\medskip

In this paper, we are studying the long-standing question whether or not there
are Landsberg spaces which are not of Berwald type from the conformal transformation point of view. 

For a Finsler manifold $(M,F)$, the conformal transformation of $F$ is defined by 
$$ \overline{F}=e^{\sigma(x)}F,$$
where $\sigma(x)$ is a function on the manifold $M$. Under conformal transformation, we  calculate the Berwald and Landsberg tensors in terms of T-tensor.

    Hashiguchi \cite{hashiguchi}, showed   that a  Landsberg space remains  Landsberg by any conformal change if and only if the T-tensor vanishes identically. However, we show that there are     Landsberg spaces (with non vanishing  T-tensor) remain Landsberg under  conformal transformation.   Matsumoto \cite{hbfinsler}  obtained that a Berwald space $(M,F)$ remains Berwald by any conformal transformation if and only if $\displaystyle{ B^{ir}_{jkh}:=\frac{\partial^3 (F^2g^{ir})}{\partial y^j\partial y^k\partial y^h}=0}$ ($F^2g^{ir}$ are quadratic in $y^i$).  However, we show that there are  Berwald spaces (with non vanishing  $B^{ir}_{jkh}$)  remain Berwald under conformal transformation.

\medskip

  Starting by Berwald space $(M,F)$, if the transformed space $(M,\overline{F})$ is Landsberg, then  a necessary condition for $(M,\overline{F})$ to be Berwald is given. This condition depends on  geometric objects of the original space $(M,F)$; like, the T-tensor $T_{hijk}$, the vertical curvature $S^{\,\, h}_{i \, \, jk}$ of Cartan connection, the Cartan tensor $C_{ijk}$ and the function $\sigma(x)$. Some special cases are considered; for example, when  $(M,F)$ is $S_3$-like  or has vanishing T-tensor and vanishing  vertical curvature of Cartan connection.

\medskip

The condition for a Berwald space to transform to  Landsberg space is $\sigma_rT^r_{jkh}=0$. We study some special Finsler spaces if they can admit a function $\sigma(x)$ such that $\sigma_rT^r_{jkh}=0$.   For example,  a positive definite   C-reducible Finsler space does not admit a function $\sigma(x)$ such that $\sigma_rT^r_{jkh}=0$.   Hence, a regular Randers space   does not admit a function $\sigma(x)$ such that $\sigma_rT^r_{jkh}=0$.

  We investigate  a case in which a  Landsberg  space can not be Berwaldian. Namely, under the conformal transformation, a  $C_2$-like Berwald space with vanishing T-tensor transforms to  non  Berwaldian Landsberg  space.  Finally, various examples are studied.  Examples of non Berwadian Landsberg (singular) spaces are shown.

   \section{conformal transformation}

 Let  $M$ be  an n-dimensional smooth manifold.   Let $(x^i)$ be
the coordinates of any  point of the base manifold $M$  and $(y^i)$ a
supporting element at the same point.  We mean by $T_xM$ the tangent space at $x\in M$ and by $TM=\undersymm{\bigcup}{x\in M}\,\,T_{x}M$ the tangent bundle of $M$.  A Finsler structure on $M$ is defined as follows: 

\begin{definition}\label{fin.struc.} A Finsler structure on a manifold $M$  is a   function
$$F:TM\rightarrow \mathbb{R}$$
with the following properties:
\begin{description}
   \item[(a)] $F\geqslant 0$ and $F(x,y)=0$ if and only if $y=0$.

    \item[(b)] $F$ is $C^\infty$ on the slit tangent
    bundle  $\tm:=TM\backslash\{0\}$.

    \item[(c)] $F(x,y)$ is positively homogenous of degree one in $y$: $F(x,\lambda y) = \lambda F(x,y)$
     for all $y \in TM$ and $\lambda > 0$.

    \item[(d)] The  matrix ${g_{ij}(x,y):=\frac{1}{2}\dot{\partial}_i\dot{\partial}_j F^2}$
is positive-definite at each point of $\tm$, where $\dot{\partial}_i$ is the  partial differentiation
    with respect to  $y^i$.
\end{description}
The pair $(M,F)$ is called  a {Finsler space} and the tensor $g=g_{ij}(x,y)dx^i\otimes dx^j$ is called the
Finsler metric  tensor  of the Finsler space $(M,F)$.
\end{definition}

  In this paper we consider the conformal transformation of a Finsler structure $F$, namely,

  \begin{equation}\label{conformal_F}
  \overline{F}=e^{\sigma(x)}F,
  \end{equation}
  where $\sigma(x)$ is smooth function on $M$.
  
   It should be noted that all geometric objects associated with the transformed space $(M,\overline{F})$ will be elaborated by barred symbols. For example, the metric tensor of   $(M,\overline{F})$ is denoted by $\overline{g}_{ij}$. And the geometric objects associated with the  space $(M,{F})$ will be elaborated by nothing. 
  
\begin{lemma}
Under the conformal transformation \eqref{conformal_F}, we have the following 
\begin{description}
\item[(a)] $\overline{\ell}_i=e^\sigma \ell_i$,

\item[(b)] $\overline{g}_{ij}=e^{2\sigma} g_{ij}$,

\item[(c)] $\overline{g}^{ij}=e^{-2\sigma} g^{ij}$,
\end{description}
where $\ell_i:=\dot{\partial}_iF$ is  the     normalized supporting element and $g^{ij}$ is the inverse  metric tensor.
\end{lemma}

 The Cartan  tensor $C_{ijk}$  is defined by $ C_{ijk}:=\frac{1}{2}\dot{\partial}_kg_{ij}$ and $C^h_{ij}:=C_{ijk}g^{kh}$. The following lemma shows the partial differentiations of $C^h_{ij}$  and other tensors obtained from the Cartan tensor with respect to $y^h$.
\begin{lemma}\label{partialdot_C}
The  following identities are useful for subsequent use:
\begin{description}
  \item[(a)] $\dot{\partial}_hC^{ir}_{j}=C^{ir}_{jh}-2C^r_{sj}C^{is}_{h}-2C^i_{sj}C^{rs}_{h}$,
  \item[(b)] $\dot{\partial}_hC^{r}_{sj}=C^{r}_{sjh}-2C_{\ell sj}C^{r\ell}_{h}$,
  \item[(c)] $\dot{\partial}_hC^{ir}_{jk}=C^{ir}_{jkh}-2C^r_{sjk}C^{is}_{h}-2C^i_{sjk}C^{rs}_{h}$,
\item[(d)]   $\dot{\partial}_hC^{r}_{ijk}=C^{r}_{ijkh}-2C^{rs}_{h}C_{ijkh}$,
\end{description}
where, $C_{ijkh}=\dot{\partial}_hC_{ijk}$, $C_{\ell ijkh}=\dot{\partial}_hC_{\ell ijk}$, $C^{r}_{ijk}=C_{\ell ijk}g^{\ell r}$, and so on.
\end{lemma}

For a Finsler metric $F = F(x, y)$ on a manifold $M$, the spray $S = y^i\frac{\partial}{\partial x^i} -2G^i \frac{\partial}{\partial y^i}$ is a vector field on $T M$, where  the functions $G^i = G^i(x, y)$ are homogeneous of degree $2$ in $y$ and called  geodesic coefficients.  $G^i$ are defined  by
$$G^i= \frac{1}{4}g^{ih}( y^r\partial_r\dot{\partial}_hF^2 - \partial_hF^2),$$
where ${\partial}_i$ is the  partial differentiation     with respect to  $x^i$. The  non linear connection $G^i_j$ and the coefficients of Berwald connection $G^i_{jk}$ are defined, respectively, by 
$$G^i_j=\dot{\partial}_jG^i, \quad G^i_{jk}=\dot{\partial}_kG^i_j.$$
By \cite{hashiguchi}, one can obtain the following lemma which shows the  transformations of   $G^i$,  $G^i_j$ and $G^i_{jk}$.
\begin{lemma}\label{chage_G}
Under the conformal transformation  \eqref{conformal_F}, we have the following 
\begin{description}
\item[(a)]$\overline{G}^i=G^i+B^i,$

\item[(b)]$\overline{G}^i_j=G^i_j+B^i_j,$

\item[(c)]$\overline{G}^i_{jk}=G^i_{jk}+B^i_{jk},$
\end{description}
where, 
$$B^i=\sigma_0y^i -\frac{1}{2}L^2\sigma^i,\qquad \sigma_0:=\sigma_iy^i,$$
$$ B^i_j=\sigma_jy^i+\sigma_0\delta^i_j -F\sigma^i\ell_j+F^2\sigma_rC^{ir}_j,$$
$$B^i_{jk}=\sigma_j\delta^i_k+\sigma_k\delta^i_j -\sigma^ig_{jk}+2F\sigma_rC^{ir}_k\ell_j+2F\sigma_rC^{ir}_j\ell_k+F^2\sigma_r(C^{ir}_{jk}-2C^r_{sj}C^{is}_{k}-2C^i_{sj}C^{rs}_{k}).$$
\end{lemma}

The T-tensor plays an important role in Finsler geometry. For example it was proved by  Szabo Zoltan \cite{Szabo_T-tensor} that a positive definite Finsler metric with vanishing T-tensor is Riemannian. Also,  Hashiguchi \cite{hashiguchi}, showed   that a  Landsberg space remains  Landsberg by any conformal change if and only if the T-tensor vanishes identically. In this paper we show that  the (strong) relation between the long-existence problem of non Berwaldian Landsberg spaces and the T-tensor. 
The T-tensor is defined by \cite{ttensor}
\begin{equation} \label{T-tensor}
T_{rijk}=FC_{rijk}-F(C_{sij}C^{s}_{rk}+C_{sjr}C^{s}_{ik}+C_{sir}C^{s}_{jk})
+C_{rij}\ell_k+C_{rik}\ell_j +C_{rjk}\ell_i+C_{ijk}\ell_r.
\end{equation}

The T-tensor is totally symmetric in all of its indices. Let us write
\begin{equation}\label{C-change14}
T_{ij}:=T_{ijhk}g^{hk}=FC_i|_j+\ell_iC_j+\ell_jC_i,\qquad T:=T_{ij}g^{ij}.
\end{equation}

\begin{lemma}\label{C-T-tensor}
The tensor $C^{ir}_{jkh}$ can be rewritten in following form
 \begin{align*}
% \nonumber to remove numbering (before each equation)
C^{ir}_{jkh}=&
  \frac{1}{F}\dot{\partial}_hT^{ri}_{jk}+\frac{2}{F}(T^{r}_{sjk}C^{is}_{h}+T^{i}_{sjk}C^{sr}_{h})-\frac{1}{F}(C^{ri}_{j}\ell_{kh}+C^{ri}_{k}\ell_{jh} +C^r_{jk}\ell^i_h+C^i_{jk}\ell^r_h)\\
  &+  \frac{1}{F}(C^i_{sj}C^r_{\ell k}+ C^r_{sj}C^i_{\ell k}+C^{ri}_{s}C_{\ell jk})\ell_h
  -2(C^i_{sj}C^r_{\ell k}C^{s\ell}_{h}+ C^r_{sj}C^i_{\ell k}C^{s\ell}_{h}+   C^{ri}_{s}C_{\ell jk}C^{s\ell}_{h}) \\
   &
  +C^{i}_{sjh}C^{sr}_{k}+C^i_{sj}C^{sr}_{kh}+C^{r}_{sjh}C^{si}_{k}+C^r_{sj}C^{si}_{kh}+C^{ir}_{sh}C^s_{jk}+C^{ri}_{s}C_{jkh}\\
  & 
   -\frac{1}{F}(C^{ir}_{jh}\ell_k+C^{ir}_{kh}\ell_j+C^{ir}_{jk}\ell_h+C^{r}_{jkh}\ell^i+C^{i}_{jkh}\ell^r).
\end{align*}
  \end{lemma}
  \begin{proof}
  Differentiating \eqref{T-tensor} with respect to $y^h$ and then raising the indecies $i$ and $r$, then using Lemma \ref{partialdot_C} we get the required  formula for $C^{ir}_{jkh}$.
  \end{proof}

   For a Berwald manifold $(M,F)$, all tangent spaces $T_xM$ with the induced Minkowski norm $F_x$ are linearly isometric. 
 There are lots of characterizations of Berwald spaces. To give one  of these characterizations,  the Berwald tensor $G^i_{jkh}$ is  defined by
  \begin{equation}\label{Berwald_tensor}
  G^i_{jkh}=\dot{\partial}_hG^i_{jk}.
  \end{equation}
 
  \begin{definition}
A Finsler manifold $(M,F)$ is said to be Berwaldian if the Berwald tensor $G^{h}_{ijk}$ vanishes identically.
\end{definition}
 
 It is known that on a Landsberg manifold $M$, all tangent spaces $T_xM$ with the
induced Riemannian metric $ \mathbf{g}_x = g_{ij}(x, y)dy^i\otimes dy^j$ are isometric. To show one of the characterizations of Landsberg space,   the  Landsberg tensor $L_{jkh}$ is  defined by
  \begin{equation}\label{Landsberg_tensor}
  L_{jkh}=-\frac{1}{2}F\ell_i G^i_{jkh}.
  \end{equation}
  
  \begin{definition}
A Finsler manifold $(M,F)$  is called Landsbergian if the Landsberg tensor $L_{ijk}$ vanishes identically.
\end{definition}

   Straightforward but long calculations  and using  Lemmas \ref{partialdot_C}, \ref{chage_G} and \ref{C-T-tensor}, we have the  following proposition.
  \begin{proposition}
Under the conformal transformation \eqref{conformal_F}, the Berwald tensor \eqref{Berwald_tensor} transforms as follows
$$\overline{G}^i_{jkh}=G^i_{jkh}+B^i_{jkh},$$
where

\begin{align*}
% \nonumber to remove numbering (before each equation)
   B^i_{jkh}=& F^2\sigma_rC^{ir}_{jkh}+2\sigma_r(C^{ir}_jg_{kh}+C^{ir}_hg_{jk}+C^{ir}_kg_{jh})
   +2F\sigma_r(C^{ir}_{jk}\ell_{h}+C^{ir}_{hj}\ell_{k}+C^{ir}_{kh}\ell_{j}) \\
   & -4F\sigma_r((C^r_{sj}C^{si}_k+C^i_{sj}C^{sr}_k)\ell_h+(C^r_{sj}C^{si}_h+C^i_{sj}C^{sr}_h)\ell_k+(C^r_{sk}C^{si}_h+C^i_{sk}C^{sr}_h)\ell_j) \\
   &  -2F^2\sigma_r((C^r_{sjk}C^{si}_h+C^i_{sjk}C^{sr}_h)+(C^r_{sjh}C^{si}_k+C^i_{sjh}C^{sr}_k)+(C^r_{skh}C^{si}_j+C^i_{skh}C^{sr}_j))\\
   &+4F^2\sigma_r((C_{stj}C^{ti}_k+C_{stk}C^{ti}_j)C^{sr}_h
   +(C_{stj}C^{it}_h+  C_{sth}C^{it}_{j})C^{sr}_k   +  (  C_{stk}C^{it}_h+ C_{sth}C^{it}_k)C^{sr}_j)
\end{align*}

  In terms of T-tensor, the Berwald tensor is given by
  \begin{align}\label{B-T-tensor}
% \nonumber to remove numbering (before each equation)
\nonumber   B^i_{jkh}=& F\sigma_r\dot{\partial}_hT^{ri}_{jk}-\sigma_r(T^{ri}_{jh} \ell_k+T^{ri}_{kh}\ell_j-T^{ri}_{jk}\ell_h-T^r_{jkh}\ell^i-T^i_{jkh}\ell^r)\\
\nonumber   &-F\sigma_r(T^i_{sjh}C^{sr}_{k}+T^r_{skh}C^{si}_{j}+T^r_{sjh}C^{si}_{k}+T^i_{skh}C^{sr}_{j}-T^{ri}_{sh}C^s_{jk}-T^s_{jkh}C^{ri}_{s})\\
   &+\sigma_r(C^{ri}_{j}h_{kh}+C^{ri}_{k}h_{jh}+2 C^{ir}_hh_{jk}-C^r_{jk}h^i_h-C^i_{jk}h^r_h-2 C_{jkh}h^{ir})    \\
 \nonumber  &+F^2\sigma_r[C^{t}_{hj}S^{\,\, ir}_{t\quad k}+C^{t}_{hk}S^{\,\, ri}_{t\quad j}-C^{ti}_{h}S^{\quad \, r}_{tjk}-C^{tr}_{h}S^{\quad \, i }_{tkj}-C^{ti}_{j}S^{\quad \, r}_{thk}-C^{tr}_{k}S^{ \quad \, i}_{th j}
   ],
\end{align}
where $S^{\,\, h}_{i\,\, jk}=C^{r}_{ik}C^{h}_{rj}-C^{r}_{ij}C^{h}_{rk}$ is the v-curvature of Cartan connection.
\end{proposition}
 
  \begin{corollary}\label{Landsberg}
  Under the conformal transformation \eqref{conformal_F}, the Landsberg tensor \eqref{Landsberg_tensor} is given by
  $$\overline{L}_{jkh}=e^\sigma L_{jkh}+e^{2\sigma} F\sigma_rT^r_{jkh}.$$
  
   \end{corollary}

 \section{Necessary condition}

  Making use of Corollary \ref{Landsberg}, one can see  easily  that under  the conformal transformation \eqref{conformal_F},  Landsberg space remains  Landsberg if and only if 
   $$\sigma_r T^r_{jkh}=0.$$

   Hashiguchi \cite{hashiguchi}, showed   that a  Landsberg space remains  Landsberg by any conformal change if and only if the T-tensor vanishes identically. However, there are     Landsberg spaces (with non vanishing  T-tensor) remain Landsberg under  conformal transformation.     This can be shown by the following example:
   
   \begin{example}
$M= \mathbb{R}^3$. Let  $\overline{F}=\sigma(x^2)F$  and $F$ defined by
$${F}(x,y)=\left(\left(y^1y^3+y^3\sqrt{(y^1)^2+(y^3)^2}\right)(y^2)^2\right)^{\frac{1}{4}}.$$
In this example, for the space $(M,F)$, we have $\sigma_rT^r_{ijk}=\sigma_2T^2_{ijk}=0 $
but, generally, $T^h_{ijk}\neq 0$ for example, $T^1_{111}\neq 0$.
\end{example}
   
   Actually, when a Landsberg space   remains    Landsberg, under any conformal transformation, then according to Hashiguchi's result the T-tensor vanishes. But using the result of Szabo Zoltan \cite{Szabo_T-tensor}, the space will be Riemannian.  So Hashiguchi's result  in this form  gives no hope to find regular Landeberg space under the conformal transformation point of view. But what will be more beneficial is to consider the case when some conformal transformations of Landsberg space preserve the Landsberg property.  
 It is worthy  to pay attention to the conformal transformation of Landsberg space which can produce a regular Landsberg space which is not  Berwald.      Now, the long existing problem of regular Landsbergian  non Berwaldian spaces  is (strongly) related to   the question:\\
 
 \textit{Is there  any regular Finsler space admitting a function $\sigma(x)$ such that $\sigma_rT^r_{ijk}=0$?}\\

   It is easy to observe  that a Berwald space remains to be Berwald if and only if  the functions $F^2g^{ir}$ are quadratic in $y^i$; that is, the tensor   
    $$B^{ir}_{jkh}:=\dot{\partial_j}\dot{\partial_k}\dot{\partial_h}(F^2g^{ir})$$
  vanishes identically. It should be noted that Asanov   and  Kirnasov \cite{Asanov_T-tensor} have calculated the tensor $B^{ir}_{jkh}$ which is related to the tensor $B^{i}_{jkh}$ by 
  $$B^{i}_{jkh}=-\frac{1}{2}B^{ir}_{jkh}\sigma_r.$$

 Also,  Matsumoto \cite{hbfinsler} (Page 786, Corollary 4.1.2.1) obtained that a Berwald space $(M,F)$ remains Berwald by any conformal change if and only if $B^{ir}_{jkh}=0$ ($F^2g^{ir}$ are quadratic in $y^i$).  However, there are  Berwald spaces (with non vanishing  $B^{ir}_{jkh}$)  remain Berwld under conformal transformation.  This can be shown by the following example:
  
   \begin{example}
$M= \mathbb{R}^4$. Let  $\overline{F}=\sigma(x^1,x^3)F$  and $F$ defined by
$${F}(x,y)=\left(\sqrt{y^1y^2}y^3y^4\left(y^2+y^4\right)\right)^{\frac{1}{4}}.$$
In this example, for the space $(M,F)$, we have $\sigma_rT^r_{ijk}=\sigma_1T^1_{ijk}=\sigma_3T^3_{ijk}=0 $
but, generally, $T^h_{ijk}\neq 0$ for example, 
$$T^4_{444}=\frac{3y^1(y^2)^3y^3y^4(3y^2+y^4)}{F^3\sqrt{y^1y^2}(3(y^2)^2+2y^2y^4+2(y^4)^2)^2}.$$
Also, we have $B^i_{jkh}=0$, $\sigma_rB^{ir}_{jkh}=\sigma_1B^{i1}_{jkh}=\sigma_3B^{i3}_{jkh}=0$, but generally $B^{ir}_{jkh}\neq 0$, for example
$$B^{44}_{444}=\frac{768y^2(y^4)^3\left((y^2)^3+(y^2)^2y^4-3y^2(y^4)^2-2(y^4)^3\right)}{\left(3(y^2)^2+2y^2y^4+2(y^4)^2\right)^4}$$
\end{example}
   
   \begin{remark}
   Throughout, most of the calculations of the examples are done by using Maple program and the Finsler package \cite{CFG}. For the simplicity reasons, most of the examples in this paper are not necessarily regular Finsler spaces but at least   non Riemannian. 
 \end{remark}
 
 Starting by a Berwald space $(M,F)$ admitting a non constant function $\sigma(x)$ such that $\sigma_rT^r_{ijk}=0$. Under the conformal transformation $\overline{F}=e^{\sigma(x)}F$, the space $(M,\overline{F})$ is Landsberg. But in order to be Berwaldian  it has to satisfy some necessary conditions. In what follow, we will try to figure out  what kind of conditions the space should satisfy.

 \begin{theorem}\label{TheoremA}
  Let   $(M,F)$ be  a Berwald space admitting a non constant function $\sigma(x)$ such that $\sigma_rT^r_{ijk}=0$. Under   the conformal transformation \eqref{conformal_F},  a necessary condition for  the Landsberg space  $(M,\overline{F})$ to be  Berwaldian   is 
  \begin{equation}\label{Berwald_condition}
  ((n -2)C^{r}+F^2C^uS^{\, \,\, r}_{u} -FT_{uv}C^{uvr} -T\ell^r)\sigma_r=0,
  \end{equation}
  where $S_{ik}:=S_{ i j k \ell}g^{j\ell}=S_{ j i \ell k}g^{j\ell}$ is the Ricci tensor of the vertical curvature.
   \end{theorem}
   
\begin{proof}
Let  $(M,F)$ be a Berwald space, then the Berwald tensor $G^i_{jkh}$ vanishes and so  the Landsberg tensor $L_{jkh}$ is. Now, contracting  \eqref{B-T-tensor} by ${g}^{kh}$, we have
  \begin{align*}
% \nonumber to remove numbering (before each equation)
  0=&-\sigma_rT^i_{j}\ell^r -F\sigma_r(T^i_{sjh}C^{srh}+T^i_{s}C^{sr}_j-T^s_{j}C^{ri}_{s}) +\sigma_r(n-2)C^{ir}_j  \\
   &+F^2\sigma_r[C^{tu}_{j}S^{\,\, ir}_{t\quad u}+C^{t}S^{\,\, ri}_{t\quad j}-C^{tiu}S^{\quad \, r}_{tju}-C^{tur}S^{\quad \, i }_{tuj}-C^{ti}_{j}S^{\, r}_{t}-C^{tur}S^{ \quad \, i}_{tu j}] 
   \end{align*}
   Contracting the above equation by  ${g}^{ij}$, using the facts that $C^{hij}S_{ijk\ell}=0$ and $S^{ \quad \, i}_{tu i}=0$ (this because $S_{ijkh}$ is antisymmetric in the first two indices  and the last two indices), we have
     \begin{align*}
% \nonumber to remove numbering (before each equation)
  0&=-\sigma_rT\ell^r -F\sigma_rT_{sh}C^{srh} +\sigma_r(n-2)C^{r}+F^2\sigma_rC^{t}S^{\, \,\, r}_{t}.
   \end{align*}
   This ends the proof.
\end{proof}   
  \begin{definition}
 A Finsler space $(M,F)$ is said to be  $S_3$-like if the vertical curvature of Cartan connection can be written in the form
\begin{equation}\label{s3-likedef.}
S_{ijkh}=\rho (h_{ik}h_{jh}-h_{ih}h_{jk}),
\end{equation}
 where $\rho:=\frac{S}{(n-1)(n-2)}$, $S=S_{ijkh}g^{jh}g^{ik}$  is the vertical scalar curvature.
 \end{definition}
Using Equation \eqref{s3-likedef.} and Theorem \ref{TheoremA}, we have the following corollary.    
   \begin{corollary}\label{Corollary_A}
 Let  $(M,F)$ be  an $S_3$-like Berwald space admitting a non constant function $\sigma(x)$ such that $\sigma_rT^r_{ijk}=0$. Under   the conformal transformation \eqref{conformal_F},  a necessary condition for  the Landsberg space  $(M,\overline{F})$ to be  Berwaldian    is 
  \begin{equation}\label{Berwald_condition}
  ((n -2)(1+F^2\rho)C^{r} -FT_{sh}C^{srh} -T\ell^r)\sigma_r=0.
  \end{equation}
   \end{corollary}
    A direct consequence of Theorem \ref{TheoremA}, we have the following result.
    \begin{theorem}
   Let  $(M,F)$ be  a  Berwald space with vanishing T-tenor and vanishing v-curvature. If the Landsberg space $(M,\overline{F})$ is Berwald, then   either $n=2$ or $\sigma_rC^r=0$.
   \end{theorem}

\section{The condition $\sigma_rT^r_{ijk}=0$}

  % Since for any three dimensional Finsler space, the vertical curvature $S^i_{jkh}$ of Cartan connection can be put on the form \eqref{s3-likedef.}. Then we have the following corollary.
 %\begin{corollary}
 %  Let  $(M,F)$ be  an three dimensional Berwald space  admitting a non constant function $\sigma(x)$ such that $\sigma_rT^r_{ijk}=0$. Under   the conformal transformation \eqref{conformal_F},  the condition \eqref{Berwald_condition} is necessary condition for   $(M,\overline{F})$ to be  Berwald.
 %  \end{corollary}

       By the above section, we showed that a Berwald space transforms to Landsberg  if the T-tensor vanishes or $\sigma_rT^r_{jkh}=0$. In this section, we focus our attention to  the condition  $\sigma_rT^r_{jkh}=0$. It should be noted that the advantages of the condition $\sigma_rT^r_{jkh}=0$ can be seen in the following manner;  it can be satisfied in regular Finsler spaces  and  it is, clearly,  a weaker condition than the   vanishing of the T-tensor.

   \begin{theorem}
   A positive definite   C-reducible Finsler space does not admit a function $\sigma(x)$ such that $\sigma_rT^r_{jkh}=0$.
   \end{theorem}
  
   \begin{proof}
   Let  $(M,F)$ is C-reducible, then T-tensor is given by
   $$T_{hijk}=\frac{T}{(n^2-1)}(h_{hi}h_{jk}+h_{ij}h_{hk}+h_{jh}h_{ik}).$$
   Contracting the above equation by $\sigma^h$, we have
   $$\frac{T}{(n^2-1)}\sigma^h(h_{hi}h_{jk}+h_{ij}h_{hk}+h_{jh}h_{ik})=0.$$
   Again, by contraction by $g^{jk}$, we get
    $$\sigma^h\frac{T}{(n^2-1)}((n-1)h_{hi}+2h_{ih})=0.$$
    Since the metric  is positive definite, then $T\neq 0$ and hence
      $$\sigma^hh_{hi}=0,$$
which gives $\sigma_i-\frac{\sigma_0}{F}\ell_i=0$. By differentiation with respect to $y^j$,  we get $\frac{\sigma_0}{F^2}\ell_{ij}=0$ which is a contradiction. In other words, $\sigma_i$ can not be proportional to the supporting element $y^i$.
   \end{proof}
   Since the metrics of Randers type are C-reducible, we have the following corollary.
 \begin{corollary}
   A regular Rander space   does not admit a function $\sigma(x)$ such that $\sigma_rT^r_{jkh}=0$.
   \end{corollary}
   
   \begin{proposition} \label{sigma-condition}
 Let $(M,F)$ be a non Riemannian space  admitting a function $\sigma(x)$ such that $\sigma_rT^r_{jkh}=0$. If  $\sigma_rC^r_{jk}=0$, then $\sigma$ is constant.
 \end{proposition}
 
 \begin{proof}
Let  $(M,L)$ be a non Riemannain space admitting $\sigma(x)$ such that $\sigma_rT^r_{jkh}=0$ and $\sigma_rC^r_{jk}=0$.
 Since $\sigma$ satisfies $\sigma_rC^r_{jk}=0$, then $\sigma^r$ is a function of $x$ only; $\dot{\partial}_j\sigma^r=-2\sigma_hC^{hr}_j=0$. Now, contracting  \eqref{T-tensor} by $\sigma^r$ then we have
 $$F^2\sigma^r C_{rijk}+\sigma_0C_{ijk}=0,$$
 which can be written in the form
  $$F^2\sigma^r \dot{\partial}_kC_{rij}+\sigma_0C_{ijk}=0.$$
  By making use of the fact that $\sigma^r(x)$ is a function of the position only, then we get $\sigma_0C_{ijk}=0$. Since the space is not Riemannian, then $\sigma_0=0$ which yields $\sigma_0=0$ and hence $ \dot{\partial}_j\sigma_0=\sigma_j=0$. Consequently, $\sigma$ is constant.
   \end{proof}

  \begin{definition}
A Finsler space $(M,L)$ of dimension $n\geq 2$ is said to be
$C_2$-like  if the Cartan  tensor $C_{ijk}$ satisfies
\begin{equation}\label{c2-likedef.}
C_{ijk}=\frac{1}{C^2}C_iC_jC_k,
\end{equation}
where $C_k:=C_{ijk}g^{ij}$ and $C^2:=C_iC^i$.
\end{definition}
The following theorem gives a case in which a Landsberg space can not be Berwaldian.
\begin{theorem}
Under the conformal transformation \eqref{conformal_F}, a  $C_2$-like Berwald space  admitting a non constant function $\sigma(x)$ such that
$$\sigma_rT^r_{ijk}=0, \quad F^2T_{uv}C^{uvr}\sigma_r +T\sigma_0=0$$
 transforms to non  Berwaldian Landsberg space.
\end{theorem}
  \begin{proof}
  Since $\sigma_rT^r_{ijk}=0$, then the space $(M,\overline{F})$ is Landsberg. For a $C_2$-like space the v-curvature vanishes and hence by using Theorem  \ref{TheoremA} and the condition $(FT_{uv}C^{uvr} +T\ell^r)\sigma_r=0$, we get 
  $$C^{r}\sigma_r=0.$$
  But making use of \eqref{c2-likedef.}, we get
  $$\sigma^i C_{ijk}=0.$$
  Now, using Proposition \ref{sigma-condition}, the space is   Riemannian or  $\sigma$ is constant and this is a contradiction.
  \end{proof}
 Since the vanishing of the T-tensor means that the space is not regular Finsler space, the following corollary  shows   a case in which a singular Landsberg space can not be Berwaldian.
\begin{corollary}
Under the conformal transformation \eqref{conformal_F}, a  $C_2$-like Berwald space  with vanishing T-tensor
 transforms to non  Berwaldian Landsberg space.
\end{corollary}
\section{Examples}

In this section we give different examples. The following two examples are obtained  by the help of \cite{Shen_example}. They give    simple classes  of  Landsberg spaces which are not Berwald.

\medskip

\begin{example} Let
$M= \mathbb{R}^3$. Let  $\overline{F}=\sigma(x^3)F$, where $F$   defined by
$${F}(x,y)=\sqrt{(y^3)^2+y^1y^2+y^3\sqrt{y^1y^2}}\,e^{\frac{1}{\sqrt{3}}\arctan\Big(\frac{2y^3}{\sqrt{3y^1y^2}}+\frac{1}{\sqrt{3}}\Big)}.$$
For $(M,F)$, 
we have,  $\sigma_rT^r_{ijk}=\sigma_3T^3_{ijk}=0$, $\sigma_rC^r=\sigma_3C^3\neq 0$,  but generally $T^h_{ijk}\neq 0$, for example, $T^1_{111}\neq 0$. Moreover, the space $(M,F)$ does not satisfy the condition \eqref{Berwald_condition}.
\end{example}

\medskip

\begin{example} Let
$M= \mathbb{R}^3$. Take $\overline{F}=\sigma(x^2)F$  and $F$ defined by
$${F}(x,y)=\sigma(x^2)\sqrt{(y^1)^2+(y^2)^2+(y^2)^3+y^2\sqrt{(y^1)^2+(y^3)^2}}\,e^{\frac{1}{\sqrt{3}}\arctan\Big(\frac{2y^2}{\sqrt{3((y^1)^2+(y^3)^2)}}+\frac{1}{\sqrt{3}}\Big)}.$$
For $(M,F)$, 
we have,  $\sigma_rT^r_{ijk}=\sigma_2T^2_{ijk}=0$, $\sigma_rC^r=\sigma_2C^2\neq 0$,  but generally $T^h_{ijk}\neq 0$, for example, $T^1_{111}\neq 0$. Moreover, the space $(M,F)$ does not satisfy the condition \eqref{Berwald_condition}.
\end{example}

\medskip

The following example shows the conformal transformation of non Berwald space produces a Berwald space.

 \begin{example} Let
$M= \mathbb{R}^3$. Take $\overline{F}=e^{\sigma(x^1,x^2)}F$  and $F$ defined by
$$F(x,y):=\sqrt[4]{\left((y^1)^2+(y^2)^2\right)^2+e^{-2\sigma(x^1,x^2)}(y^3)^4}.$$
In this example $\overline{G}^i_{jkh}=0$ but  ${G}^i_{jkh}=-B^i_{jkh}$. Also,   $\overline{L}_{jkh}= 0$ and  ${L}_{jkh}=- e^{\sigma(x^1,x^2)}F\sigma_rT^r_{jkh}$. 
\end{example}

% \begin{example} Let
%$M= \mathbb{R}^3$. Take $\overline{F}=f(x^2)F$  and $F$ defined by
%$$F(x,y):=\left(\left(y^1y^2+y^2\sqrt{(y^1)^2+(y^3)^2}\right)(y^1)^2\right)^{\frac{1}{4}}.$$
%In this example $T^2_{ijk}=0$ but generally $T^h_{ijk}\neq 0$, for example, $T^1_{111}\neq 0$. So we have $\sigma_hT^h_{ijk}=\sigma_2T^2_{ijk}=0$,$\sigma_2C^2\neq 0$.
%\end{example}

%\begin{example}\label{ex.3} Let
%$M= \mathbb{R}^3$. Take $\overline{F}=\sigma(x^1)F$  and $F$ defined by
%$$F(x,y):=\left(\left(y^1y^2+y^1\sqrt{y^2y^3}\right)(y^3)^2\right)^{\frac{1}{4}}$$
%\end{example}


\begin{thebibliography}{99}



\bibitem{Asanov_T-tensor}
G. S. Asanov and E. G. Kirnasov, \emph{On finsler spaces satisfying theT-condition},  Aequationes Mathematicae \textbf{24}, (1982), 66--73.

\bibitem{Asanov}
G. S. Asanov, \emph{Finsleroid-Finsler spaces of positive-definite and relativistic types}, Rep.  Math. Phys., \textbf{58} (2006), 275--300.

\bibitem{Shen_etal}
S. B{\'a}cs{\'o},   X. Cheng and Z. Shen,  \emph{Curvature properties of ($\alpha$, $\beta$)-metrics}, Adv. Stud. Pure Math., \textbf{48} (2007), 73--110.

\bibitem{Bao}
D. Bao,  \emph{Unicorns in Finsler geometry}, in Proceedings of the 40th Symposium on Finsler
Geometry (Hokkaido Tokai University, Sapporo, Japan, 2005)






\bibitem{hashiguchi}
M.~Hashiguchi, \emph{On conformal transformations of  Finsler
metrics},
  J. Math. Kyoto Univ., \textbf{16} (1976), 25--50.



%\bibitem{Shen-book} S.S. Chern and Z. Shen: \emph{Riemann-Finsler Geometry},
  %World Scientific\blue{,  2004}.





\bibitem{Landsberg2}
G. Landsberg,  \emph{ {\"U}ber die Kr{\"u}mmungstheorie und Variationsrechnung},  Jahresber. Deutsch. Math. Verein., \textbf{16} (1907),  547--557.

\bibitem{Landsberg1}
G. Landsberg,  \emph{{\"U}ber die Totalkr{\"u}mmung},  Jahresber. Deutsch. Math. Verein., \textbf{16} (1907),  36--46. 

 

\bibitem{Landsberg3}
G. Landsberg,  \emph{ber die Kr{\"u}mmung in der Variationsrechnung},  Math. Ann.,  \textbf{65} (1908),  313--349. 


\bibitem{r2.9}
M.~Matsumoto, \emph{On C-reducible Finsler spaces }, Tensor, N.
S., \textbf{24} (1972), 29--37.

\bibitem{ttensor}
M.~Matsumoto, \emph{V-transformations of  Finsler spaces. I.
Definition, infinitesimal transformations and isometries},
 J.
Math. Kyoto Univ., \textbf{12} (1972),
  479--512.

%\bibitem{r2.8}
%M.~Matsumoto, \emph{On Finsler spaces with Randers metric and
%special  forms of important tensors}, J. Math. Kyoto Univ., \textbf{14} (1974),  477--498.

\bibitem{matsumoto}
 M. Matsumoto,\emph{\ On three dimensional Finsler spaces
 satisfying the $T$ and $B^p$ conditions},
 Tensor, N. S., \textbf{29} (1975), 13--20.

\bibitem{hbfinsler} M. Matsumoto in:  Handbook of Finsler
geometry II, ed. P.~L. Antonelli (Kluwer Acad. publ., 2003).

\bibitem{Matsumoto_conformal}
M. Matsumoto, \emph{Conformally closed   Finsler
spaces},
 Balkan J. Geom. Appl., \textbf{4}(1999), 117--128.


\bibitem{shen-book1} Z. Shen: \emph{Differential geometry of spray and Finsler
    spaces}, Springer, 2001.

\bibitem{Shen_example}
Z. Shen, \emph{On a class of Landsberg metrics in Finsler geometry}, Canad. J.  Math., \textbf{61} (2009), 1357--1374.


\bibitem{Szabo_T-tensor}
Z. Szabo, \emph{Positive definite Finsler spaces satisfying the T-condition are Riemannian  }, Tensor, N.
S., \textbf{35} (1981), 247--248.




\bibitem{Szilasi-book} J. Szilasi, R.L. Lovas, D.Cs. Kert\'esz:
  \emph{Connections, Sprays and Finsler Structures}, World Scientific, 2014.



\bibitem{CFG}
Nabil~L. Youssef and  S. G. Elgendi, \emph{New Finsler   package},  Comput. Phys. Commun., \textbf{185} (2014),  986--997.
ArXiv: 1306.0875  [math. DG].




%\bibitem{semi-cvf}
%Nabil L. Youssef, S. G. Elgendi and  E. Taha, \emph{Semi concurrent vector field},
 % Under preparation (2017).
  
  \end{thebibliography}
\end{document}